\documentclass[12pt]{amsart}
\usepackage{fourier}
\usepackage[T1]{fontenc}
\usepackage{graphics}
\usepackage{amsfonts,amssymb,color}
\usepackage[mathscr]{eucal}
\usepackage{amsmath, amsthm}
\usepackage{mathrsfs}
\usepackage{amsbsy}
\usepackage{dsfont}
\usepackage{bbm}
\usepackage{wasysym}
\usepackage{stmaryrd}
\usepackage{url}

\input xypic
\xyoption{all}

\makeindex
\makeglossary

\begin{document}
\baselineskip = 16pt

\newcommand \ZZ {{\mathbb Z}}
\newcommand \NN {{\mathbb N}}
\newcommand \RR {{\mathbb R}}
\newcommand \PR {{\mathbb P}}
\newcommand \AF {{\mathbb A}}
\newcommand \GG {{\mathbb G}}
\newcommand \QQ {{\mathbb Q}}
\newcommand \CC {{\mathbb C}}
\newcommand \bcA {{\mathscr A}}
\newcommand \bcC {{\mathscr C}}
\newcommand \bcD {{\mathscr D}}
\newcommand \bcF {{\mathscr F}}
\newcommand \bcG {{\mathscr G}}
\newcommand \bcH {{\mathscr H}}
\newcommand \bcM {{\mathscr M}}
\newcommand \bcI {{\mathscr I}}
\newcommand \bcJ {{\mathscr J}}
\newcommand \bcK {{\mathscr K}}
\newcommand \bcL {{\mathscr L}}
\newcommand \bcO {{\mathscr O}}
\newcommand \bcP {{\mathscr P}}
\newcommand \bcQ {{\mathscr Q}}
\newcommand \bcR {{\mathscr R}}
\newcommand \bcS {{\mathscr S}}
\newcommand \bcV {{\mathscr V}}
\newcommand \bcU {{\mathscr U}}
\newcommand \bcW {{\mathscr W}}
\newcommand \bcX {{\mathscr X}}
\newcommand \bcY {{\mathscr Y}}
\newcommand \bcZ {{\mathscr Z}}
\newcommand \goa {{\mathfrak a}}
\newcommand \gob {{\mathfrak b}}
\newcommand \goc {{\mathfrak c}}
\newcommand \gom {{\mathfrak m}}
\newcommand \gon {{\mathfrak n}}
\newcommand \gop {{\mathfrak p}}
\newcommand \goq {{\mathfrak q}}
\newcommand \goQ {{\mathfrak Q}}
\newcommand \goP {{\mathfrak P}}
\newcommand \goM {{\mathfrak M}}
\newcommand \goN {{\mathfrak N}}
\newcommand \uno {{\mathbbm 1}}
\newcommand \Le {{\mathbbm L}}
\newcommand \Spec {{\rm {Spec}}}
\newcommand \Gr {{\rm {Gr}}}
\newcommand \Pic {{\rm {Pic}}}
\newcommand \Jac {{{J}}}
\newcommand \Alb {{\rm {Alb}}}
\newcommand \Corr {{Corr}}
\newcommand \Chow {{\mathscr C}}
\newcommand \Sym {{\rm {Sym}}}
\newcommand \Prym {{\rm {Prym}}}
\newcommand \cha {{\rm {char}}}
\newcommand \eff {{\rm {eff}}}
\newcommand \tr {{\rm {tr}}}
\newcommand \Tr {{\rm {Tr}}}
\newcommand \pr {{\rm {pr}}}
\newcommand \ev {{\it {ev}}}
\newcommand \cl {{\rm {cl}}}
\newcommand \interior {{\rm {Int}}}
\newcommand \sep {{\rm {sep}}}
\newcommand \td {{\rm {tdeg}}}
\newcommand \alg {{\rm {alg}}}
\newcommand \im {{\rm im}}
\newcommand \gr {{\rm {gr}}}
\newcommand \op {{\rm op}}
\newcommand \Hom {{\rm Hom}}
\newcommand \Hilb {{\rm Hilb}}
\newcommand \Sch {{\mathscr S\! }{\it ch}}
\newcommand \cHilb {{\mathscr H\! }{\it ilb}}
\newcommand \cHom {{\mathscr H\! }{\it om}}
\newcommand \colim {{{\rm colim}\, }} % colimit
\newcommand \End {{\rm {End}}}
\newcommand \coker {{\rm {coker}}}
\newcommand \id {{\rm {id}}}
\newcommand \van {{\rm {van}}}
\newcommand \spc {{\rm {sp}}}
\newcommand \Ob {{\rm Ob}}
\newcommand \Aut {{\rm Aut}}
\newcommand \cor {{\rm {cor}}}
\newcommand \Cor {{\it {Corr}}}
\newcommand \res {{\rm {res}}}
\newcommand \red {{\rm{red}}}
\newcommand \Gal {{\rm {Gal}}}
\newcommand \PGL {{\rm {PGL}}}
\newcommand \Bl {{\rm {Bl}}}
\newcommand \Sing {{\rm {Sing}}}
\newcommand \spn {{\rm {span}}}
\newcommand \Nm {{\rm {Nm}}}
\newcommand \inv {{\rm {inv}}}
\newcommand \codim {{\rm {codim}}}
\newcommand \Div{{\rm{Div}}}
\newcommand \CH{{\rm{CH}}}
\newcommand \sg {{\Sigma }}
\newcommand \DM {{\sf DM}}
\newcommand \Gm {{{\mathbb G}_{\rm m}}}
\newcommand \tame {\rm {tame }}
\newcommand \znak {{\natural }}
\newcommand \lra {\longrightarrow}
\newcommand \hra {\hookrightarrow}
\newcommand \rra {\rightrightarrows}
\newcommand \ord {{\rm {ord}}}
\newcommand \Rat {{\mathscr Rat}}
\newcommand \rd {{\rm {red}}}
\newcommand \bSpec {{\bf {Spec}}}
\newcommand \Proj {{\rm {Proj}}}
\newcommand \pdiv {{\rm {div}}}
\newcommand \wt {\widetilde }
\newcommand \ac {\acute }
\newcommand \ch {\check }
\newcommand \ol {\overline }
\newcommand \Th {\Theta}
\newcommand \cAb {{\mathscr A\! }{\it b}}

\newenvironment{pf}{\par\noindent{\em Proof}.}{\hfill\framebox(6,6)
\par\medskip}

\newtheorem{theorem}[subsection]{Theorem}
\newtheorem{conjecture}[subsection]{Conjecture}
\newtheorem{proposition}[subsection]{Proposition}
\newtheorem{lemma}[subsection]{Lemma}
\newtheorem{remark}[subsection]{Remark}
\newtheorem{remarks}[subsection]{Remarks}
\newtheorem{definition}[subsection]{Definition}
\newtheorem{corollary}[subsection]{Corollary}
\newtheorem{example}[subsection]{Example}
\newtheorem{examples}[subsection]{examples}

\title{Bloch's conjecture on surfaces of general type with an involution}
\author{Kalyan Banerjee}

\address{Indian Institute of Science Education and Research, Mohali, India}

\email{kalyanb@iisermohali.ac.in}

\begin{abstract}
In this short note we prove that the Bloch's conjecture holds for a surface of general type of numerical Godeaux  type or some class of numerical Campedelli type,  with geometric genus zero equipped with an involution, when the quotient of the surface by the involution is a rational surface.
\end{abstract}
\maketitle

\section{Introduction}

One of the most important problems inn algebraic geometry is to compute the Chow groups for higher codimensional cycles on a smooth projective variety. The theorem due to Mumford in \cite{M}, proves that the Chow group of zero cycles is infinite dimensional provided that the geometric genus of the surface is greater than zero, in the sense that the natural maps from symmetric powers to the Chow group are never surjective. The converse is a conjecture due to Bloch, that if we have a smooth projective complex surface of geometric genus zero then the Chow group of zero cycles is isomorphic to the group of integers.
This theorem has been proved for surfaces not of general type with geometric genus zero due to \cite{BKL}. The case for surfaces of general type that is surfaces of Kodaira dimension 2 with geometric genus zero is still open. Some examples of such surfaces for which the Bloch's conjecture holds are due to \cite{B}, \cite{IM}, \cite{PW},\cite{V},\cite{VC}. This paper concerns the proof of Bloch' conjecture for some examples of surfaces of general type with geometric genus zero equipped with an involution. There is a generalization of the Bloch's conjecture which says that if we have an involution on a surface of general type which acts as identity on the space of global two forms then such an involution acts as identity on the group of algebraically trivial zero cycles on the surface modulo rational equivalence (denoted as $A_0(S)$). This is the departing point of proving  the Bloch's conjecture on a surface of general type with an involution. The technique due to \cite{Voi}, tells us that if we have  such an involution on a K3 surface then the involution acts as identity on the $A_0$ of the K3 surface. This has also been proved for some examples of K3 surfaces due to \cite{GT},\cite{HK}. So we prove that the involution acts as identity on the group $A_0(S)$, for some examples of surfaces  $S$ of general type with geometric genus equal to zero, such that the quotient is rational and equipped with an involution. Since the  quotient surface $S/i$ is rational, it has trivial $A_0$ proving that the involution is acting as $-1$ on $A_0(S)$. These two informations together tell us that the group $A_0(S)$ is trivial by the Roitman's torsion theorem \cite{R2}.

The main difference of our approach to the Bloch's conjecture on a surface of general type with an involution with that of Voisin present in \cite{Voi}, is the mild generalisation of the monodromy argument, for the rest of the arguments we mainly follow the idea of \cite{Voi} and \cite{BKL}. The argument present in \cite{Voi} involves the monodromy of a Lefschetz pencil over the field of complex numbers, whereas we present the monodromy technique over an arbitrary uncountable algebraically closed ground field, so that we can get rid of the analytic argument and obtain an arithmetic variant of it. These arithmetic monodromy first appears in \cite{Del1}, in the proof of Weil conjectures and a very elaborated explanation of it is written in \cite{FK}. This mild generalizations enables us to understand the Bloch's conjetcure over uncountable algebraically closed fields.

So we arrive at the following main result in this paper:
\begin{theorem}
Let $S$ be a numerical Godeaux surface over the field of complex numbers with an involution $i$ on it such that the quotient of $S$ by the involution is a rational surface with effective anticanonical divisor, such that the branch locus either consists of a unique irreducible component of genus $2$, and the other components of the branch locus are rational, then the Bloch's conjecture holds for this surface $S$.
\end{theorem}

Also the Bloch's conjecture holds for certain examples of numerical Campedelli surfaces with an involution. Namely for those such that the quotient is rational or elliptic with all the components of the branch locus being rational.

{\small \textbf{Acknowledgements:} The author would like to thank the hospitality of IISER-Mohali and Tata Institute Mumbai, for hosting this project. The author is indebted to V. Srinivas and N.Fakhruddin for constructive criticism about the proof of the main theorem of this paper.}

\section{Finite dimensionality in the sense of Roitman and \'etale monodromy}

Let us recall the finite dimensionality in the sense of Roitman. Let $X$ be a smooth projective variety and let $A_0(X)$ be the group of algebraically trivial zero cycles on $X$ modulo the rational equivalence. A subgroup $P$ of $A_0(X)$ is said to be finite dimensional in the sense of Roitman, if there exists a smooth projective variety $W$, and a correspondence $\Gamma$ of the correct codimension on $W\times X$, such that P is contained in the set $\Gamma_*(W)$. Let $Z$ be a correspondence supported on $S\times S$, $S$ is a smooth projective surface, then following Voisin \cite{Voi} we have the following proposition.

\begin{proposition}
Suppose that the image of $Z_*:A_0(S)\to A_0(S)$ is finite dimensional. Then $Z_*$ factors through the albanese map $alb_S:A_0(S)\to Alb(S)$.
\end{proposition}

\begin{proof}
The proof of this is same as in \cite{Voi} but we present it for the sake of completeness of the argument. The only variant present in our proof would be the presence of \'etale monodromy instead of usual monodromy of the topological fundamental group.

Since the image of $Z_*$ is finite dimensional, there exists a smooth projective variety $W$ and a correspondence $\Gamma$ on $W\times S$, such that image of $Z_*$ is contained in $\Gamma_*(W)$. Consider a very ample line bundle $L$ on $S$, and embed $S$ into the corresponding projective space. Let $C$ be a smooth hyperplane section of $S$ inside this projective space. We can take some multiple of $L$, so that the hyperplane section $C$ is of high genus. So that the kernel of the map from $J(C)$ to $Alb(S)$ denoted by $K(C)$ has dimension greater than $W$. Now we prove the following lemma.

\begin{lemma}
The abelian variety $K(C)$ is simple for a very general $C$.
\end{lemma}

\begin{proof}
This also follows from generalizing the argument of \cite{CV}[corollary 1.2]. We include our proof for the convenience of the reader. Let $A$ be a non-trivial proper abelian subvariety of $K(C)$ for a  very general $C$. Then $A$ is defined over $k$, so we can extend the scalars to $k(t)$, and consider $A_{k(t)}$. Attaching the co-efficients of the defining equation of $A_{k(t)}$ and $K(C)_{k(t)}$ we have $A_{k(t)}, K(C)_{k(t)}$ defined over a finite extension $L$ of $k(t)$  inside $\overline{k(t)}$. Let us choose a Lefschetz pencil $D$  on $S$, such that its geometric generic fiber is $C$. Then let $D'$ be a smooth projective curve mapping finitely onto $D$ and having the function field equal to $L$. Now spread $A_{k(t)},K(C)_{k(t)}$ over $U'$ Zariski open in $D'$ and let us denote the spreads by $\bcA,\bcK$ respectively. Throwing out more points from $U'$ we get that the morphisms $\bcA, \bcK \to U'$ are smooth and proper. Let these two morphisms be $f,g$. Then consider the constant sheaves corresponding to the pre-sheaves given by $\ZZ/l^n \ZZ$ on $\bcA,\bcK$ respectively. Since the morphisms $f,g$ are smooth and proper, we have that $R^if_*\ZZ/l^n\ZZ, R^ig_*\ZZ/l^n\ZZ$ are locally constant \'etale sheaves on $U'$. Hence taking a geometric point $\eta'$ over the geometric point corresponding to $C$ (let us denote it by $\eta$), we have that $\pi_1(U',\eta')$ acts on the stalk of the above sheaves,
$$(R^if_*\ZZ/l^n\ZZ)_{\eta'}, (R^ig_*\ZZ/l^n\ZZ)_{\eta'}$$
 respectively. So cosidnering $i=1$, we get that $\pi_1(U',\eta')$ acts on
$$H^1_{\'et}(A_{\eta'},\ZZ/l^n \ZZ), H^1_{\'et}(K(C),\ZZ/l^n \ZZ)$$

 Now the map $A_{\eta'}$ to $K(C)$ gives rise to a surjection at the level of Tate modules of the corresponding ableian varieties. So taking dual we have that $H^1_{\'et}(A_{\eta'},\ZZ_l)$ is embedded into $H^1_{\'et}(K(C),\ZZ_l)$, and this map is a map of $\pi_1$-modules because it is induced by a map of the corresponding locally conatant sheaves. Now $H^1_{\'et}(K(C),\QQ_l)$ is isomorphic to the kernel of $H^1_{\'et}(C,\QQ_l)\to H^3_{\'et}(S,\QQ_l)$, because the isomorphism of $H^1_{\'et}$ of an ableian variety with the dual of its Tate module is functorial. So now we have an action of $\pi_1(U',\eta')$ on $H^1_{\'et}(A_{\eta'},\QQ_l)$, which is embedded in $H^1_{\'et}(K(C),\QQ_l)$. Also for the image of $U'$ in $D$, denoted by $U$, we have that $\pi_1(U,\eta)$ ($\eta$ is the geometric point corresponding to $C$), acts irreducibly on $H^1_{\'et}(K(C),\QQ_l)$ by the Picard Lefschtez formula (More details about Picard-Lefschetz formula and the irreducibility of the monodromy action can be found in \cite{FK}[chapter III] and \cite{Del1}[Corollary 5.5]). Now it is a consequence of the Picard Lefschetz formula and the fact that $\pi_1(U',\eta')$ gives rise to a finite index subgroup of $\pi_1(U,\eta)$ that $H^1_{\'et}(A_{\eta'},\QQ_l)$ is $\pi_1(U,\eta)$-equivariant. Hence we have that $H^1_{\'et}(A_{\eta'},\QQ_l)$ is either $H^1_{\'et}(K(C),\QQ_l)$ or zero. Now note that $A_{\eta'}$ is $A$, so we have that the Tate module of $A$ is isomorphic to the Tate module of $K(C)$ or its Tate module is trivial. Therefore $A$ is isogenous to $K(C)$ or to the trivial abelian variety. Hence $K(C)$ is simple.
\end{proof}

Next let us consider the subset

$$R=\{(k,w)\in K(C)\times W|Z_*j_*(k)=\Gamma_*(w)\}$$
then by Mumford-Roitman argument as in \cite{M},\cite{R}, we have that $R$ is a countable union of Zariski closed subsets of the product. Since image of $Z_*$ is finite dimensional we have that $R$ maps surjectively onto $K(C)$. So there exists a component $R_0$ of $R$ which maps surjectively onto $K(C)$. So we have
$$\dim(R_0)\geq \dim(K(C))> \dim(W)\;.$$
Let $w$ be a closed point on $W$, then the fiber of the projection over $w$ is of dimension greater or equal than one. Denote the fiber by $F_w$, since it is of dimension greater or equal than one, it generates the whole abelian variety $K(C)$. So it follows that $Z_*$ vanishes on $K(C)$.

Next, let us consider a cycle $z=z^+-z^-$ on $S$ of degree zero. That will correspond to a tuple $(s_1,\cdots,s_{2k})$ on $S^{2k}$. We blow up $S$ at these points and obtain $S'$ and a map $\tau:S'\to S$. Let $E_i$ be the exceptional divisor on $S'$ corresponding to $s_i$. Let $H$ be a very ample divisor on $S$ such that $\tau^*(H)-n\sum_i E_i$ is ample on $S'$. Then consider a sufficiently high multiple of $L$, which is very ample, the for any curve $C$ in the corresponding linear system we have that $\tau(C)$ contains the points $s_i$. Now let $alb_S(z)=0$, that would imply that any lift of $z$ say $z'$ on $C$, will go to zero under the albanese map of $S'$. So then we have $z'$ supported on $K(C)$. We can repeat the arguments as above for the correspondence $Z'=Z\circ \tau$, it is finite dimensional, hence $Z'_*=0$ on $K(C)$, which in turn implies that $Z_*(z)=Z'_*(z')=0$. So we have that $Z_*$ factors through the albanese map of $S$.
\end{proof}

Now we consider the correspondence $Z=\Delta_S-Graph(i)$. An we prove that the image of $Z_*$ is finite dimensional.

\begin{theorem}
\label{theorem2}
Let $S$ be a numerical Godeaux surface having an involution $i$ on it. Let $S/i$ be a rational surface. Let the fixed locus of $i$ contains a unique irreducible component of genus $2$.  Let us consider the correspondence $\Delta_S-Gr(i)$ on $S\times S$. Then the image of the push-forward induced by this correspondence is finite dimensional.
\end{theorem}

\begin{proof}

The fixed locus of the involution on $S$ is the union of a reducible curve $R$ and finitely many isolated points such that

$$R=Z_1\cup \cdots Z_n\cup \Gamma$$
where $Z_i$'s are $-2$ rational curves and $\Gamma$ is smooth of genus $2$.  Let us blow up the surface $S$ at these finitely many  isolated fixed points to get a surface $S'$. $S'$ carries a natural involution on it. If we quotient by this involution we have the desingularization of the quotient surface $S/i$, which is rational. Then the branch locus of the involution on $S'$ is a union of $R$ and $\cup_i E_i$, where $E_i$'s are the exceptional curves of the Blow-up. Consider the blow down of $S'$ along the rational components of $R$ and along $E_i$'s. Call this surface $S''$. There is an involution $i$ on $S''$ such that the quotient is the blow down of $S'/i$ along the rational components of the fixed locus of the involution on $S'$.  Then we take a very ample line bundle $L$ on the surface $S'/i=\Sigma$. Let the genus of a curve in the linear system of $L$ be $g$. We consider the exact sequence of sheaves
$$0\to\bcO(C)\to \bcO(\Sigma)\to \bcO(\Sigma)/\bcO(C)\to 0$$
then tensoring with $\bcO(-C)$ we get that
$$0\to \bcO(\Sigma)\to \bcO(-C)\to \bcO(-C)|_C\to 0\;.$$
This gives rise to the exact sequence of global section of sheaves
$$0\to \CC\to H^0(\Sigma, L)\to H^0(C,L|_C)\to 0$$
this sequence is exact on the right because the irregularity of $\Sigma$ is zero. Now by the ampleness of $L$, we have that the degree of $L|_C$ is greater than zero, so we have by Riemann-Roch
$$\dim(H^0(C,L|_C))=g+n-1\;.$$

Here  is $n$ large by using the fact that the rational surfaces occur are having anticanonical divisor effective \cite{CCL}[proposition 3.9 and remark 4.13], also see \cite{H}[introduction for examples of anticanonical rational surfaces] and we can choose $L$ to be of sufficiently large degree.

Hence the dimension of $H^0(\Sigma,L)$ is equal to $g+n$.

Now for any $C$ in the linear system of $L$, we have its pre-image on $S'$, say $\wt{C}$ is a two sheeted cover of $C$. It is smooth for a general $C$ and connected by the Hodge index theorem. Suppose that $C_1,C_2$ be two components of $\wt{C}$, such that $C_1^2>0,C_2^2>0$ and $C_1.C_2=0$, because $\wt{C}$ is smooth. But then by the Hodge index theorem we have that the intersection form restricted to $C_1,C_2$ is semipositive, which can only happen when $C_1$ is proportional to $C_2$, hence we have that $C_1^2=0=C_2^2$, which further gives $(C_1+C_2)^2=0$, contradicting the ampleness of $\wt{C}$.
So we have that $\wt{C}$ is a connected double cover of $C$. But this double cover could be ramified along the intersection of $\wt{C}$ with the branch locus $B=R\cup \cup_i E_i$ of the involution. By \cite{CCL} we have that the number

$$(2K_{\Sigma}+B).L\geq 0$$
for $L$ ample. Consider the unique component $\Gamma$ of $R$ which is of genus $2$. By \cite{CCL} proposition $4.5$ and corollary 4.8 we have $B-\Gamma=-2K_{\Sigma}$. So we have
$$B-2\Gamma=-2K_{\Sigma}-\Gamma\;.$$
To prove that
$$(-2K_{\Sigma}-\Gamma).L>0$$
we have to prove that
$$(B-2\Gamma).L>0\;.$$
But this could be achieved since the curve $\Gamma$ on $S$ is coming from a quartic curve on $P^2$, which is a component of the branch locus of the double over of $\PR^2$. The branch locus of the double cover of $\PR^2$ we consider is a degree $10$ curve having a quartic curve as a irreducible component.

So consider the map $\pi: S\to \PR^2$, consider $\pi(R),\pi(\Gamma)$ on $\PR^2$. By Bezout's theorem
$$(\pi(B)-\pi(2\Gamma)).L>0$$
for $L$ a very ample line bundle on $\PR^2$. Let us consider the blow up of $\PR^2$ along finitely many points to achieve $\Sigma$. Let $E'_i$'s are the exceptional curves of the blow up. This gives us that
$$(B-2\Gamma).(\pi^*L-\sum_i E'_i)=\pi^*(\pi(B)-\pi(2\Gamma)).\pi^*(L)=(\pi(B)-\pi(\Gamma)).L>0\;.$$
Therefore choosing $L$ to be multiple of the pull-back of a very ample line bundle on $\PR^2$, we get the required inequality.

\medskip

Let us consider a general point in $S''^{g+n}$, say $(s_1,\cdots,s_{g+n})$.
Let $(\sigma'_1,\cdots,\sigma'_{g+n})$ be the image of this point in $\Sigma'=S''/i$ under the quotient map. Then there exists a unique curve $C$ in $L$ such that its image under the blow down  contains all the points $\sigma'_i$, and its double cover contains $s_i$ for all $i$. Then the image of $\sum_i s_i-\sum_i i(s_i)$, depends on the element
$$alb_{\wt{C}}(\sum_i s_i-i(s_i))$$
which is an element in  the Prym variety of the double cover $\wt{\pi'(C)}\to \pi'(C)$. Here $\pi'$ is the blow down map from $\Sigma$ to $\Sigma'$. So  the map from $S''^{g+n} $ to $A_0(S'')$ factors through the Prym fibration $\bcP(\wt{\pi'{\bcC}}/\pi'{\bcC})$, where $\pi'{\bcC}$ is the universal family of curves in $\Sigma'$ parametrized by the linear system $|L|$, and $\wt{\pi'{\bcC}}$ is its universal double cover. The genus of the double cover of $\pi'(C)$ for a general $C$ is equal to
$$g'=2g-1+m+\Gamma.L/2$$
here $m$ is the number of points obtained by blowing down the rational curves and
$$\Gamma.L$$
is the intersection number of $\pi'(\wt{C})$ with $\Gamma$ (this is obtained by projection formula). So the dimension of the prym variety of the corresponding double cover is
$$g-1+m+ \Gamma.L/2\;.$$

 Hence the dimension of $\bcP(\wt{\pi'{\bcC}}/\pi'{\bcC})$ is equal to $g-1+g+n+m+\Gamma.L/2=2g+n-1+m+\Gamma.L/2$, whereas the dimension of $S''^{g+n}$ is $2g+2n$. So the fibers of the map
$$S''^{g+n}\to \bcP(\wt{\bcC}/\bcC)$$
are of dimension
$$2g+2n-(2g+n-1+m+\Gamma.L/2)=n-m+1-\Gamma.L/2$$.

Consider $n$ such that $n-\Gamma.L/2>m$ and by previous calculation we know that
$$n-\Gamma.L/2>0\;.$$
SO the fibers of the above map are positive dimensional. Hence it contains a curve $F_s$. Now the linear system corresponding to $5K_S$ is very ample\cite{Bo}(since $S$ is minimal. Consider $D=5K_{S'}-(\sum_i E_i)=\pi^*(5K_S)$, and $K_{S}^2$ lies in between $1$ to $2$. So we have the genus of the curves $D$ in the linear system of $5K_{S'}-\sum_i E_i$, a fixed number equal to
$$\pi^*(25K_S^2)+\pi^*(5K_S).(\pi^*(K_S)+\sum_i E_i)=25K_S^2+5K_S^2=2g-2.$$

So $$g=15K_S^2+1\;,$$ which gives us that $g=31$ or $16$ and that is
less than $g+n$. Consider the divisor $\pi'(D)$ on $S''$, it is ample by the Nakai-Moisezhon criterion and by the projection formula.
$\sum_i \pr_i^{-1}(\pi'(D))$ on the product $S''^{g+n}$. This divisor is ample so it intersects $F_s$, so we get that there exist points in $F_s$ which are  in $\pi'(D)\times S''^{g+n-1}$ where $D$ is in the linear system of $L$. Then consider the elements of $S''^{g+n}$ of the form $(c,s_1,\cdots,s_{g+n-1})$, where $c$ belongs to $\pi'(D)$. Then such a given $(c,s_1,\cdots,s_{g+n-1})$ lies on a unique $\wt{\pi'(C)}.$  Consider the Zariski open  subset of $S''^{g+n}$, consisting of $(c,s_1,\cdots,s_{g+n-1})$ where all $s_i$'s are distinct. Project it down to $S''^{g+n-1}$.  Then we can consider  the map from $S''^{g+n-1}$ to $A_0(S'')$ given by
$$(s_1,\cdots,s_{g+n-1})\mapsto alb_{\wt{C}}(\sum_i s_i-\sum_i i(s_i))\;,$$
we see that this map factors through the Prym fibration and the map from $S''^{g+n-1}$ to $\bcP(\wt{\pi'(\bcC)}/\pi'(\bcC))$ has positive dimensional fibers, since $n$ is large. So it means that, if we consider an element $(s_1,\cdots,s_{g+n-1})$, such that $(c,s_1,\cdots,s_{g+n-1})$ is in $F_s$ and a curve through it, then it intersects the ample divisor given by $\sum_i \pr_i^{-1}(\pi'(D))$, on $S''^{g+n-1}$. Then we have some of $s_i$ is contained in $C$. So iterating this process we get that, there exists elements of $F_s$ that are supported on $\pi'(D^k)^k\times S''^{g+n-k}$, where $k$ is some natural number depending on $n$. But the genus of $\pi'(D)$ is fixed and is less than or equal to $16$ or $31$. So if we choose $n$ large enough, then $k$ will be large and greater than the genus of $\pi'(D)$. This means that the elements of $F_s$ are supported on $\pi'(D)^{n_0}\times S''^{g+n-k}$, where $n_0$ is the genus of $\pi'(D)$. Then all we have is that elements of $F_s$ are supported on $S''^{m_0}$, where $m_0$ is strictly lesser than $g+n$.

 Hence we have that the cycle on $F_s$ is supported on $S''^{m_0}$, where $m_0$ is strictly less than $g+n$. So we have that $\Gamma_*(S''^{m_0})=\Gamma_*(S''^{g+n})$, where $\Gamma=\Delta_S''-Graph(i)$. This proves that the image of $\Gamma_*$ is finite dimensional by the following argument.

Now we prove by induction that $\Gamma_*(S''^{m_0})=\Gamma_*(S''^m)$ for all $m\geq g+n$.
So suppose that $\Gamma_*(S''^k)=\Gamma^*(S''^{m_0})$ for $k\geq g+n$, then we have to prove that $\Gamma_*(S''^{k+1})=\Gamma_*(S''^{m_0})$. So any element in $\Gamma_*(S''^{k+1})$ can be written as  $\Gamma_*(s_1+\cdots+s_{m_0})+\Gamma_*(s)$. Now let $k-m_0=l$, then $m_0+1=k-l+1$. Since $k-l<k$, we have $k-l+1\leq k$, so $m_0+1\leq k$, so we have the cycle
$$\Gamma_*(s_1+\cdots+s_{m_0})+\Gamma_*(s)$$
supported on $S''^k$, hence on $S''^{m_0}$. So we have that $\Gamma_*(S''^{m_0})=\Gamma_*(S''^k)$ for all $k$ greater or equal than $g+n$. Now any element $z$ in $A_0(S'')$, can be written as a difference of two effective cycle $z^+,z^-$ of the same degree. Then we have
$$\Gamma_*(z)=\Gamma_*(z^+)-\Gamma_*(z_-)$$
and $\Gamma_(z_{\pm})$ belong to $\Gamma_*(S''^{m_0})$. So let $\Gamma'$ be the correspondence on $S''^{2m_0}\times S''$ defined as
$$\sum_{l\leq m_0}(pr_i,pr_{S''})^*\Gamma-\sum_{m_0\leq l\leq 2m_0}(pr_i,pr_{S''})^* \Gamma$$
where $\pr_i$ is the $i$-th projection from $S''^{m_0}$ to $S''$, and $\pr_{S''}$ is from $S''^i\times S''$ to the last copy of $S''$. Then we have
$$\im(\Gamma_*)=\Gamma'_*(S''^{2m_0})\;.$$

Therefore the image of $(\Delta_{S''}-\Gr(i_{S''}))_*$ has finite dimensional image in the sense of Roitman.  Since the image of  $(\Delta_{S''}-\Gr(i_{S''}))_*$ is mapped surjectively onto the image of $(\Delta_{S'}-\Gr(i_{S'}))_*$, it has finite dimensional image implying that $(\Delta_S-i_S)_*$ has finite dimensional image.

%\textit{The Enriques case}:

%\medskip

%Consider a point in $S'^{ng}=S'^g\times \cdots \times S'^g$. Then such a point is contained in an $n$-fold product of curves $C_1\times \cdots \times C_n$, where $C_i$'s are in the linear system of $L$ and they are unique. Then we have the cycle $\sum_{i=1}^{ng}s_i-i(s_i)$, belongs to the n-fold product $P(\wt{C_i}/C_i)$ of the Prym varieties. So the map from $S'^{ng}$ to $A_0(S')$ is factoring through $\prod_i \bcP(\wt{\bcC}/\bcC)$. Hence the fibers of this map from $S'^{ng}$ to $\bcP(\wt{\bcC}/\bcC)$ are positive dimensional and are of dimension $n$. Also observe that the fibers of this map is product of fibers of the individual maps from $S'^g$ to $\bcP(\wt{\bcC}/\bcC)$. Therefore as before consider an ample divisor on $S'$, whose general member is of genus less than $ng$ (say $L$ as before). Then we have that fibers are supported on $S^{ng-n}\times C^n$, where $C$ is in the linear system of $L$. Now consider $n$ to be larger than the genus of $C$, then we have that the image of $\Gamma_*(S^{ng})$ equals the image of $\Gamma_*(S^{ng-n}\times C^{m_0})=\Gamma_*(S^{n_0})$, where $n_0$ is less than $ng$, here $\Gamma=\Delta-Graph(i)$. Then the finite dimensionality follows from similar argument as in the previous case.

\end{proof}
Since the irregularity of $S$ is zero we have that,
\begin{theorem}
The involution $i_*$ acts as identity on the group $A_0(S)$.
\end{theorem}
%\begin{remark}
%In the above proof we only used that the surface $\Sigma$ has irregularity zero. But for later purpose we need the surface $\Sigma$ to be rational or Enriques or Elliptic.
%\end{remark}

\section{The involution acts as -1 on $A_0(S)$}
Since the quotient $S/i$ of the given surface is birational to an Enriques surface or a rational surface, we have that the $i_*$-invariant part in $A_0(S)$ is trivial. Hence $i_*$ acts as $-\id$ on $A_0(S)$. We prove something more in  this section. Namely we are going to prove that if $S$ such that its quotient by the involution is birational to an Elliptic surface, then the involution acts as -1 on $A_0(S)$. For this we follow the Bloch-Kas-Lieberman technique as in \cite{BKL}.

\begin{theorem}
Let $S$ be a surface of general type with $p_g=q=0$. Let $i$ be an involution on $S$, such that $S/i$ is birational to an elliptic surface. Then $i_*$ acts as $-\id$ on $A_0(S)$.
\end{theorem}

First consider the pencil of elliptic curves on the  surface $\Sigma$, that is a regular morphism from $\Sigma\to D$, where $D$ is a smooth projective curve.  Let $f$ denote the morphism from $S'$ to $\Sigma$. Then consider the general fiber of the elliptic fibration $\Sigma\to D$ and further pull it back to $S'$ via $f$. Let us denote $f^{-1}(E)$ by $C$ for a general elliptic curve $E$ of the fibration $\Sigma\to D$.  Consider the Jacobian fibration $J\to D$ corresponding to $\Sigma\to D$. Now fix a multisection $Y$ of $\Sigma$ and let $\pi$ be the morphism from $J\to D$. Let us have
$$Y\cap \pi^{-1}(t)=\sum_{i=1}^n p_i(t)$$
for $t$ in $D$. We map $\Sigma$ to $J\to D$ by
$$q\mapsto(q,\cdots,q)\mapsto nf(q)-\sum_{i=1}^n p_i(\pi(q))$$
so  we have a dominant morphism $g$ from $S'$ to $J$.

\begin{lemma}
\label{prop3}
Let $T(J)$ denote the albanese kernel for $J$. If $T(J)=0$, then the involution $i$ acts as $-1$ on the  albanese kernel $T(S)$ for $S$.
\end{lemma}
\begin{proof}
Now want to understand the quasi-inverse of of $g$. Let $\alpha$ belong to $J$ that lies over $t\in D$. So there is a unique point in $q_{i}(t)$ on $E$ such that $q_i(t)-p_i(t)$ is rationally equivalent to $\alpha$. Now $g^{-1}(q_i(t))=\{q_i'(t),q_i''(t)\}$. So we can define $\lambda$ to be
$$\alpha\mapsto \sum_i (q_i'(t)+q_i''(t)).$$
Now we check that
$$g_*\lambda(\alpha)=g_*(\sum_i q_i'(t)+q_i''(t))$$
which is
$$=\sum_i g_*(q_i'(t))+g_*(q_i''(t))$$
$$=\sum_i 2n(q_i(t)- p_i(t))=2n^2\alpha\;.$$

Let $q$ be a point supported on $C_t=f^{-1}\pi^{-1}(t)$ where $\pi:J\to D$. Then we prove that $2n(q+iq)$ is rationally equivalent to zero on $S'$. For that let us compute $\lambda g_*(q+iq)$. So we have by definition of
$$g_*(q+iq)=nf_{*}(q+iq)-2\sum_{i=1}^n (p_i(t))\;,$$
let $f_t(q)=q'$.
Then
$$\lambda g_*(q+iq)=\lambda(2nq')-\lambda(2\sum_{i=1}^n (p_{i}(t)))$$
which can be re-written as
$$2\sum_{i=1}^n \lambda(q'-(p_{i}(t)))\;.$$
Now $\lambda(q'-(p_{i}(t)))=(q+iq)$
so putting that in the above we get that
$$2n(q+iq)=\lambda g_*(q+iq)\;.$$
Therefore for any zero cycle $z$ of degree zero on $S$, we have
$$2n(z+iz)-\lambda g_*(z+iz)=0\;.$$
Suppose $z$ vanishes on $Alb(S')$, hence $iz$ vanishes on $Alb(S')$, then we have $g_*(z+iz)$ vanishing on $Alb(J)$. If $T(J)$ is zero then we can conclude that $g_*(z+iz)$ is rationally equivalent to $0$ on $J$. Composing with $\lambda$ we get that $2n(z+iz)=0$. By Roitman's theorem \cite{R2}[theorem 3.1] it will follow that $z+iz=0$, so we have $iz=-z$.
\end{proof}

Now we repeat the proof  that $T(J)=0$ by showing that $J$ is rational following \cite{BKL}.

\begin{lemma}
$$T(J)=0$$
\end{lemma}

\begin{proof}
We have $p:J\to D$ and a section $\sigma:D\to J$ by the very definition of a Jacobian fibration. First we observe that $D=\PR^1$ as $q(J)=0$ and $Alb(J)$ maps onto $Alb(D)$. Then by \cite{Sha}[VII, section 3] we have
$$K_J=\pi^*(K_J.\sigma(D))$$
also by the adjunction formula we have
$$K_J.\sigma(D)=K_{D}-N$$
where $N$ is the normal bundle to $D$ in $J$. Thus we have
$$h^0(J,rK_J)=h^0(J,\pi^*(r(K_{D}-N)))=h^0(D,r(K_{D}-N))$$
since $p_g=0$ we have that
$$h^0(D,K_{D}-N)=0\;.$$
Since $D$ is $\PR^1$, the degree of $K_{D}-N$ is negative. Therefore by Riemann-Roch we have $h^0(D,r(K_{D}-N))=0$ for $r>0$. Therefore $P_r(J)=0$ for all $r\geq 1$, in particular $P_2=0$, so $J$ is rational. Therefore $T(J)=0$.
\end{proof}

The case when $\Sigma$ is rational is similar. First suppose that the rational surface contains a pencil of elliptic curves on it. Let $E_{\eta}$ be  the generic fiber of this pencil and pull it back to $S'$, call it $C_{\eta}$. Then spreading out the curve $C_{\eta},E_{\eta}$, over a smooth projective curve $D'$, we get that $S'$ admits a map to an elliptic surface. Then consider a multisection of this elliptic surface, and a dominant map from $S'$ to it, which ensures that the geometric genus and irregularity is zero for the elliptic surface. Then we proceed as in the proof of \ref{prop3}.

\section{Application of the above results}
Consider  a surface of general type of Numerical Godeaux type with $p_g=q=0$ having an involution $i$ on it. Suppose that the quotient by the involution is birational to a$ \PR^2$ and the fixed locus consists of a unique irreducible component of genus two. Then we have  that $i_*$ acts as $\id$ and $-\id$.  So we have that $2id=0$ on $A_0(S)$. This says that by the Roitman's torsion theorem that $id=0$ on $A_0(S)$. So we get the following theorem.

\begin{theorem}
Let $S$ be a numerical Godeaux surface over the field of complex numbers with an involution $i$ on it such that the quotient of $S$ by the involution is a rational surface with a unique irreducible component of genus $2$ of the branch locus. Then the Bloch's conjecture holds for this surface $S$.
\end{theorem}

%\begin{remark}
%The proof of \ref{theorem2} also shows that if we carry out the argument for the correspondence $\Delta+Graph(i)$, then it follows that $i_*$ acts as $-1$ on $A_0(S)$. So having an elliptic pencil on the surface is not necessary.
%\end{remark}

%\begin{example}
%Consider the numerical Godeaux surfaces with an involution, that is surfaces od general type with $p_g=q=0$ and with the self intersection of the canonical class equal to $1$. Then according to \cite{CCL} we have the following classification of such surfaces. Their quotients by the involutions are either birational to an Enriques surface or to a rational surface. In the rational surface case, there are two types rational surfaces with a pencil of elliptic curves or with a pencil of rational curves. In both the cases the surfaces are having effective anticanonical divisor, \cite{CCL}[remark 4.13]. So according to our method the Bloch's conjecture hold's for the surfaces whose quotient by the involution are  birational to Enriques surfaces or rational surfaces.
%\end{example}

Consider a numerical Campedelli surface with an involution. They are classified in the following way by \cite{CLP}. If the bi-canonical map is composed by the involution, meaning that the bicanonical map factors through the quotient then the quotient is either birational to an Enriques surface or a rational surface.  The rational case the surfaces are with effective anticanonical divisors. Some of them has all the components of the branch locus being rational. In this cases the Bloch's conjecture follow from our method. In the other case when the bicanonical map is not composed by an involution then the quotient is either birational to an Enriques surface or a rational surface and it is not of general type, or it is a numercial Godeaux surface or an elliptic surface not of general type. For the case when the quotient is  a rational surface or it is an elliptic surface (all not of general type), we have the Bloch's conjecture for the surface. Here we use the fact that the fixed locus consists of $-4$ rational curves.  The case of the proper elliptic surfaces requires that the elliptic fibration has a section. So we have the following theorems. 
\begin{theorem}
Let $S$ be a numerical Campedelli surface with an involution $i$. Suppose that the bi-canonical map is composed with the involution and the quotient is rational. Also suppose that the fixed locus of the involution consists only $-2$ rational curves. Then the Bloch's conjecture holds for the surface.
\end{theorem}
\begin{proof}
The proof follows arguing as in \ref{theorem2}. We have to blow down the $-2$ curves to have the ramification locus disjoint from an ample curve on the relevant surface. Since the ramification locus on the original surface are only $-2$ curves, the ramification locus on the blow-down consists of only points.
\end{proof}

\begin{theorem}
Let $S$ be a numerical Campedelli surface with an involution $i$. Suppose that the bi-canonical map is not composed with the involution and the quotient is rational or elliptic with a section. Then the Bloch's conjecture holds for $S$.
\end{theorem}

\begin{proof}
Proof is similar as in \ref{theorem2}. We blow down the fixed locus of the involution which consists of only $-4$ rational curves.
\end{proof}

%\begin{remark}
%Moreover our technique to prove that the involution acts as identity involves etale monodromy and is free from complex analytic techniques. So we can deduce that when the surface of general type with geometric genus zero and irregularity zero is equipped with an involution and its quotient is either elliptic, or birational to an Enriques surface or a rational surface having an effective anticanonical divisor on it, then the group $A_0(S)$ is trivial. In particular if we consider such surfaces over the algebraic closure of $\QQ$, we have the group $A_0(S)$ is trivial.
%\end{remark}


\begin{thebibliography}{AAAAA}
\bibitem[BG]{BG} K.Banerjee, V.Guletskii, {\em Rational equivalence of line configuartions of cubic hypersurfaces in $\PR^5$.}, arxiv: 1405.6430.
\bibitem[B]{B} R.Barlow, {\em Rational equivalence of zero cycles for some surfaces with $p_g=0$}, Invent. Math. 1985, no. 2, 303-308.
\bibitem[Bo] {Bo} E.Bombieri, {\em Canonical models of surfaces of general type}, Publicacions Mathematiques de IHES, Volume 42, 1973, 171-219.
\bibitem[CCL]{CCL}A.Calabri, C.Ciliberto, M.Mendes Lopez {\em Numerical Godeaux surfaces with an involution}, Transactions of the AMS, Vol. 359, No. 4, 2007, 1605-1632.
\bibitem[CCP]{CLP} A.Calabri, M.Mendes Lopez, R.Pardini {\em Involutions on numerical Campedelli surfaces}, Tohoku Math 60, (2008), 1-22.
\bibitem[CV]{CV} C.Ciliberto, G. Van Der Geer, {\em On the Jacobian of a hyperplane sections of a surface}, Classification of irregular varieties, minimal models and abelian varieties, Proceeding of a Conference, Trento, Italy, 1990, Springer.
\bibitem[BKL]{BKL}S.Bloch, A.Kas, D.Lieberman, {\em Zero cycles on surfaces with $p_g=0$}, Compositio Mathematicae, tome 33, no 2(1976), page 135-145.

\bibitem[CG]{CG} P.Craighero, R.Gattazzo, {\em Quintics surfaces in $\PR^3$ having a non-singular model with $q=p_g=0$, $P_2\neq 0$},Ren. Sem. Math. Uni. Padova, 91:187-198,1994.

\bibitem[Del1]{Del1} P.Deligne, {\em La conjecture de Weil I}, Publicacions Mathematiques de IHES, 1974, volume 43, page 273-307.
\bibitem[Del2]{Del2} P.Deligne, {\em La conejcture de Weil II}, Publicaciones Mathematiques de IHES, 1980, volume 52, page 137-252.

\bibitem[DW]{DW} I.Dolgachev, C.Werner, {\em A simply connected numerical Godeaux surface with ample canonical class.},
    {\small \tt arXiv:alg-geom/9704022}
\bibitem[FK]{FK} E.Freitag, R.Kiehl {\em Etale cohomology and Weil conjectures}, Springer, Berlin-Heidelberg-Newyork.
\bibitem[GT]{GT} V.Guletskii, A.Tikhomirov {\em Algebraic cycles on quadric sections of cubics in $\PR^4$ under the action of symplectomorphisms}, Proc. of the Edinburgh Math. Soc. 59 (2016) 377 - 392.
%\bibitem[Gul]{Gul} V.Guletskii, {\em Motivic obstruction to rationality }, arXiv:1605.09434
%\bibitem[Gul1]{Gul1}V.Guletskii, {Bloch's conjecture for surfaces with involutions and of geometric genus zero}, arXiv:1704.04187
%\bibitem[Gul2]{Gul2} V.Guletskii, {Bloch's conjecture for the surface of Craighero and Gattazzo }, arXiv:1609.04074

\bibitem[HK]{HK} D.Huybrechts, M.Kemeny, {\em Stable maps and Chow groups}, Documenta Math. 18, 2013, 507-517.
\bibitem[H]{H} B.Harbourne, {\em Anticanonical rational surfaces}, Trans. AMS, Volume 349, no. 3, 1997, 1191-1208
\bibitem[IM]{IM} H.Inose, M.Mizukami, {\em Rational equivalence of 0-cycles onn some surfaces with $p_g=0$}, Math. Annalen, 244, 1979, no. 3, 205-217.


\bibitem[M]{M} D.Mumford, {\em Rational equivalence for $0$-cycles on surfaces.}, J.Math Kyoto Univ. 9, 1968, 195-204.
\bibitem[PW]{PW} C.Pedrini, C.Weibel, {\em Some examples of surfaces of general type for which Bloch's conjecture holds}, arXiv:1304.7523.
\bibitem[R]{R} A.Roitman, {\em $\Gamma$-equivalence of zero dimensional cycles (Russian)}, Math. Sbornik. 86(128), 1971, 557-570.
\bibitem[R1]{R1}A.Roitman, {\em Rational equivalence of 0-cycles}, Math USSR Sbornik, 18, 1972, 571-588
\bibitem[R2]{R2} A.Roitman, {\em The torsion of the group of 0-cycles modulo rational equivalence}, Ann. of Math. (2), 111, 1980, no.3, 553-569
\bibitem[Sha]{Sha} I. Shafarevich, et.al,{\em Algebraic surfaces}, Proceedings of the Steklov Mathematics Institute of Mathematics, 75, 1965.

\bibitem[SV]{SV} A.Suslin, V.Voevodsky, {\em Relative cycles and Chow sheaves}, Cycles, transfers, motivic homology theories, 10-86, Annals of Math studies.
\bibitem[Voi]{Voi}C.Voisin,{\em Symplectic invoultions of K$3$ surfaces act trivially on $CH_0$}, Documenta Mathematicae 17, 851-860,2012.
\bibitem[Vo]{Vo} C.Voisin, {\em Complex algebraic geometry and Hodge theory II}, Cambridge studies of Mathematics, 2002.
\bibitem[V]{V}C.Voisin, {\em Bloch's conjecture for Catanese and Barlow surfaces}, Journal of Differential Geometry, 2014, no.1, 149-175.
\bibitem[VC]{VC} C.Voisin, {\em Sur les zero cycles de certaines hypersurfaces munies d'un automorphisme}, Ann. Scuola Norm. Sup. Pisa Cl. Sci., (4), 19, 1992, no.4, 473-492.
\end{thebibliography}
\end{document}